\documentclass[a4paper,12pt]{amsart}

\usepackage{amssymb}
\usepackage[francais]{babel}
\usepackage{mhequ}
 \usepackage[utf8]{inputenc} 
\usepackage{lmodern}
\renewcommand{\div}{\mbox{div}}

\usepackage{cite}

\newcommand{\bel}[1]{\begin{equation}\label{#1}}
\newcommand{\beal}[1]{\begin{eqnarray}\label{#1}}
\newcommand{\beadl}[1]{\begin{deqarr}\label{#1}}
\newcommand{\eeadl}[1]{\arrlabel{#1}\end{deqarr}}
\newcommand{\eeal}[1]{\label{#1}\end{eqnarray}}
\newcommand{\eead}[1]{\end{deqarr}}
\newcommand{\eea}{\end{eqnarray}}
\newcommand{\eeaa}{\end{eqnarray*}}

\newcommand{\be}{\begin{equation}}
\newcommand{\ee}{\end{equation}}

\newcounter{mnotecount}[section]

\newcommand{\N}{{\mathbb N}}

\newcommand{\rmnote}[1]{}

\newcommand{\Ric}{\operatorname{Ric}}

\def\mysavedown#1{\edef\mysubs{\mysubs#1}}
\def\mysaveup#1{\edef\mysups{\mysups#1}}
\def\mydown#1{{\mytensor}_{\vphantom{\mysubs}#1}}
\def\myup#1{{\mytensor}^{\vphantom{\mysups}#1}}
\def\tensor#1#2{
  #1
  \def\mytensor{\vphantom{#1}}
  \def\mysubs{\relax}
  \def\mysups{\relax}
  \let\down=\mysavedown
  \let\up=\mysaveup
  #2
  \let\down=\mydown
  \let\up=\myup
  #2
  }

\newcommand{\Riem}{\operatorname{Riem}}

\newcommand{\Tr}{\operatorname{Tr}}

\newcommand{\R}{\mathbb R}

\renewcommand{\div}{\operatorname{div}}

\DeclareMathOperator{\Hess}{Hess}

\renewcommand{\phi}{\varphi}
\renewcommand{\epsilon}{\varepsilon}

\def\crn#1#2{{\vcenter{\vbox{
        \hbox{\kern#2pt \vrule width.#2pt height#1pt
           }
          \hrule height.#2pt}}}}

\newcommand{\Ein}{\operatorname{Ein}}

\renewcommand{\hbar}{{\overline h}}

\newcommand{\pre}[2]{{{\vphantom{#2}}^{#1}}\kern-.2ex{#2}}

\theoremstyle{plain}
\newtheorem{theorem}{Théorème}[section]
\newtheorem{lemma}[theorem]{Lemme}
\newtheorem{proposition}[theorem]{Proposition}

\theoremstyle{definition}

\numberwithin{equation}{section}

\date{22 janvier 2017}

\begin{document}
\title[{I}nversion d'op\'erateurs de courbure II]
{{I}nversion d'op\'erateurs de courbures au voisinage d'une métrique {R}icci parall\`ele~II:  vari\'et\'es non compactes à g\'eom\'etrie born\'ee.}

\author[E. Delay]{Erwann Delay}
\address{Erwann Delay,
Avignon Universit\'e,
Labo. de Math. d'Avignon,
 F-84916 Avignon, France}
\email{Erwann.Delay@univ-avignon.fr}
\urladdr{http://www.math.univ-avignon.fr}

\begin{abstract}
On considère une
 variété riemannienne $(M,g)$ non compacte, complète,  à géométrie bornée et courbure de Ricci parallèle. Nous montrons que certains opérateurs "affines" en la courbure de Ricci sont localement
inversibles, dans des espaces de Sobolev classiques, au voisinage de $g$.
\end{abstract}


\maketitle

\noindent {\bf Mots clefs }: Variété non compacte, Courbure de Ricci, 
2-tenseurs symétriques, système elliptique quasi-linéaire, Problème inverse, espaces de Sobolev.
\\
\newline
\tableofcontents

\section{Introduction}\label{section:intro}
Sur  une variété Riemannienne $(M,g)$, considérons $\Ric(g)$ sa courbure de Ricci  et $R(g)$ sa courbure scalaire.
Parmi les (champs de) 2-tenseurs symétriques géométriques naturels que l'on peut construire,
les plus simples sont ceux qui seront "affines" en la courbure de Ricci, autrement dit, de la forme
$$
\Ein(g):=\Ric(g)+\kappa R(g)g+\Lambda g,
$$
o\`u $\kappa$ et $\Lambda$ sont des constantes.
Ainsi, si $\kappa=\Lambda=0$ on retrouve la courbure de Ricci, si $\kappa=-\frac12$  le tenseur d'Einstein (avec constante cosmologique $\Lambda$), enfin si $\kappa=-\frac1{2(n-1)}$ et $\Lambda=0$ le 
tenseur de Schouten.
Ce tenseur est géométriquement  naturel :  pour tout difféomorphisme $\varphi$ assez régulier,
$$
\varphi^*\Ein(g)=\Ein(\varphi^*g).
$$
Nous nous posons ici le problème de l'inversion de l'opérateur $\Ein$.
On se donne donc $E$ un champ de tenseur symétrique sur $M$,  on cherche $g$ métrique riemannienne  telle
\bel{mainequation}
\Ein(g)=E.
\ee
On doit ainsi résoudre un système quasi-linéaire particulièrement complexe.
La motivation d'une telle question,  ainsi qu'une liste des travaux antérieurs sur 
le sujet, sont détaillés dans 
 \cite{Delay:ricciproduit}  et ses références,  cette question y étant étudiée  sur des 
variétés compactes. 

L'objectif de cette note est de montrer que les résultats alors
obtenus sont transposables à une large classe de variétés  non compactes. 
Les preuves   identiques 
ne seront pas reproduites. Cette exposition veut faire ressortir
uniquement des ingrédients  suffisants pour répondre au problème dans ce nouveaux
contexte. Elle permettra une adaptation aisée à d'autres cadres.
Par exemple, pour des géométries  particulières, o\`u l'on veut mesurer plus précisément le comportement des fonctions ou (champs de) tenseurs  
 via des espaces à poids ({variétés asymptotiquement cylindriques},
{ asymptotiquement coniques},
{ à cusps},
{ à singularités coniques},...), 
il suffira de vérifier l'éventuelle validité des quelques étapes données ici.
Certains cas  de variétés asymptotiquement euclidiennes ou asymptotiquement hyperboliques
ayant  été analysées par le passé \cite{Delay:ricciAE, Delay:etude,Delay:study, DelayHerzlich}.\\

On considère une variété riemannienne  $(M, g)$ sans bord, complète, non compacte, lisse et  {\it Ricci parallèle}.
Nous supposons de  plus qu'elle est à {\it géométrie bornée}:  son rayon d'injectivité est minoré
(par une constante strictement positive) et  toutes les dérivées covariantes de la courbure de Riemann sont bornées.

Notre but  est de prouver un résultat d'existence locale sur $M$ près de la
métrique  $g$. Nous travaillons pour cela dans des espaces de Sobolev
classiques  $H^{k}$ de fonctions (ou champs de tenseurs, voir section
\ref{sec:Hst} pour une définition plus précise).

Un exemple de résultat que nous nous proposons de montrer ici est le suivant :
\begin{theorem}\label{maintheorem}
Soit  $s\in\N$ tels que $s>\frac n2$, $\kappa=0$. Soit $\Lambda$ un réel tel que $-2\Lambda$ n'est pas dans le spectre $L^2$ du 
Laplacien de Lichnerowicz $\Delta_L$, ou bien est simplement  dans son spectre discret.   
On suppose aussi que $-2\Lambda$ n'est pas dans le spectre $L^2$ du Laplacien de Hodge agissant sur les 1-formes.
Alors pour tout $e\in H^{s+2}(M,\mathcal S_2)$ petit,  il existe un unique $h$ proche de zéro dans 
$H^{s+2}(M,\mathcal S_2)$ telle que
$$
\Ein( g+h)+\frac12\Pi(h)=\Ein( g)+e,
$$
o\`u $\Pi(h)$ est la projection orthogonale $L^2$ de $h$ sur le noyau de $\Delta_L+2\Lambda$.
De plus l'application $e\mapsto h$ est lisse au voisinage de zéro entre les espaces de Hilbert correspondants. 
\end{theorem}
Pour $\Lambda$ assez grand toutes les conditions sont clairement vérifiées et $\Pi=0$, 
l'équation (\ref{mainequation}) est donc résolue au voisinage de $g$.

Ce théorème est un cas particulier du théorème \ref{theoinvEin} o\`u tous les $\kappa\neq -1/n$ et $\kappa\neq-1/2(n-1)$ sont aussi autorisés
à condition que la métrique $g$ soit en plus d'Einstein.
L'analogue du résultat sur la courbure de Ricci contravariante  obtenu dans 
\cite{Delay:ricciproduit} est facilement transposable dans ce nouveau contexte, 
il ne sera   pas décrit ici.

La régularité de notre solution est optimale, il suffit de transporter l'équation par un difféomorphisme
peu régulier pour s'en convaincre.

Le fait que la métrique de départ soit Ricci parallèle équivaut au fait  qu'elle est localement le produit de métriques d'Einstein
(voir par exemple \cite{Wu:Holonomy}).

Cette inversion nous permet ensuite, en section \ref{sec:ssvar}
de prouver que l'image de certains opérateurs de type
Riemann-Christoffel sont des sous-variétés dans des espaces de Fréchet.\\

{\small\sc Remerciements}. {\small
 Je remercie Gilles Carron pour les références \cite{Schubin1991} et \cite{Bar2000}. 
}

\section{Définitions, notations et conventions}\label{sec:def}

Pour une métrique riemannienne $g$, nous noterons  $\nabla$ sa connexion  de Levi-Civita, par $\Ric(g)$   sa courbure de Ricci et par
$\Riem(g)$ sa courbure de  Riemann sectionnelle. 

Soit ${\mathcal T}_p^q$ l'ensemble des tenseurs covariants de rang $p$ et contravariants de rang $q$.
Lorsque $p=2$ et $q=0$, on notera ${\mathcal S}_2$ le sous-ensemble des tenseurs symétriques,
qui se décompose en ${\mathcal G}\oplus {\mathring{\mathcal
S}_2}$ o\`u ${\mathcal G}$ est l'ensemble des  tenseurs  $g$-conformes et 
${\mathring{\mathcal S}_2}$ l'ensemble des tenseurs sans trace (relativement à $g$). On utilisera la convention de sommation  d'Einstein
 (les indices correspondants vont de $1$ à $n$), et nous utiliserons 
 $g_{ij}$ et son inverse $g^{ij}$ pour monter ou descendre les indices.

Le Laplacien (brut) est défini par
$$
\triangle=-tr\nabla^2=\nabla^*\nabla,
$$
o\`u $\nabla^*$ est l'adjoint formel $L^2$ de $\nabla$. 
Pour  $u$ un champ de 2-tenseur covariant symétrique, on définit sa divergence par
 $$ (\mbox{div}\,u)_i=-\nabla^ju_{ji}.$$ Pour une 1-forme
$\omega$ on $M$, on définit sa divergence par :
$$
d^*\omega=-\nabla^i\omega_i,
$$
et la partie symétrique  de ses dérivées covariantes:
$$
({\mathcal
L}\omega)_{ij}=\frac{1}{2}(\nabla_i\omega_j+\nabla_j\omega_i),$$
(notons que ${\mathcal L}^*=\mbox{div}$).

On définit l'opérateur de Bianchi des 2-tenseurs symétriques dans les 1-formes :
$$
B_g(h)=\div_gh+\frac{1}{2}d(\Tr_gh).
$$
Le Laplacian de
Lichnerowicz  agissant sur les (champs de) 2-tenseurs covariant symétriques est
$$
\triangle_L=\triangle+2(\Ric-\Riem),
$$
o\`u
$$(\Ric\; u)_{ij}=\frac{1}{2}[\Ric(g)_{ik}u^k_j+\Ric(g)_{jk}u^k_i],
$$
et
$$
(\Riem \; u)_{ij}=\Riem(g)_{ikjl}u^{kl}.
$$
Le laplacien de Hodge-de Rham agissant sur les 1-formes sera noté
$$
\Delta_H=dd^*+d^*d=\Delta+\Ric.
$$

\section{Outils d'analyse}\label{sec:Hst}
Les espaces que nous utiliserons sont les espace des Sobolev classique $H^k$ de fonctions ou tenseurs
ayant $k$ dérivées covariantes (au sens des distributions) dans $L^2$.
Plus précisément un champ de tenseur $u$ est  dans $H^k(M,{\mathcal T}_p^q)$ si
$u$ est dans $H^{k}_{loc}$ et, la quantité suivante, qui représentera sa norme dans $H^k$ est finie
$$
\|u\|_k=\left(\int_M\sum_{i=1}^k\|\nabla^{(i)}u\|_g^2d\mu_g\right)^{\frac12}.
$$

Sous la condition de géométrie bornée (définie en introduction), ces espaces ont beaucoup de bonnes propriétés comme :\\

L'injection de Sobolev  (voir \cite{Eichhorn2007} théorème 3.4  page 16 par exemple)
$$
s>\frac n2+k\;\Rightarrow \;\;H^{s}\subset C^k_b,
$$
o\`u $C^k_b$ est l'ensemble des fonctions (ou champs de tenseurs) $C^k$ sur $M$ dont les dérivées covariantes d'ordre $\leq k$
sont bornées. Cette injection permet entre autre de s'assurer que les champs de 2-tenseurs 
symétriques de la forme $g+h$, avec $h$ petit dans $H^s$, sont encore définit 
positifs.

Le lemme suivant a aussi son importance (voir \cite{Eichhorn2007} théorème 3.12  page 21 par exemple).
\begin{lemma}\label{lemAlgebre}
Soient $s>\frac n2$,  $u\in H^{s}(M,{\mathcal T}^p_q)$ et $v\in H^{s}(M,{\mathcal T}^k_l)$ alors on a 
$ u\otimes v\in H^{s}(M,{\mathcal T}^{p+k}_{q+l})$.
De plus il existe une constante $C$, indépendante de $u$ et $v$  telle que
$$
\|u\otimes v\|_{s}\leq C \|u\|_{s} \|v\|_{s}.
$$
\end{lemma}

Enfin, nous avons  besoin de propriétés d'isomorphismes pour des opérateurs du type $\nabla^*\nabla+$ termes de courbures,
agissant sur les champs de 2-tenseurs symétriques, sur les 1-formes, ou les fonctions.
Nous renvoyons le lecteur à des références comme \cite{Schubin1991} ou \cite{Bar2000} pour le vocabulaire
et certains outils utilisés ici. 
On considère donc un fibré tensoriel $E$ sur $M$ et 
$$P=\nabla^*\nabla+K,$$
 o\`u $K$ est un endomorphisme borné de $E$.
On suppose  que 
$$
P:H^2(M,E)\rightarrow L^2(M,E)
$$ 
est Fredholm, en particulier le noyau $L^2$ de $P$ est de dimension finie.
Nous noterons alors $\Pi$ la projection
orthogonale $L^2$ sur $\ker P$. Ainsi, si $h_1,...,h_k$ est une base $L^2$-orthonormée
de $\ker P$, 
$$
\Pi(h)=\sum_{i=1}^k\langle h,h_i\rangle_{L^2}h_i.
$$
Nous pouvons énoncer la
\begin{proposition}\label{DeltaLiso}
Soient $k\in\N$ et $c\in\R$ avec $c\neq 0$. Alors
$P+c\,\Pi$ est un isomorphisme de $H^{k+2}(M,E)$ dans $H^{k}(M,E)$.
\end{proposition}
\begin{proof}
Notons $\mathcal K$ le noyau de dimension finie de $P$, ces éléments sont lisses par régularité elliptique.
On note $\mathcal K^\perp$, l'orthogonal $L^2$  de $\mathcal K$.
Alors   $P$ étant Fredholm, 
$$P: H^2\cap\mathcal K^\perp\longrightarrow \mathcal K^\perp
$$
est un isomorphisme. 
Ensuite tout élément $h\in H^2$ se décompose en $$h=u^\perp+u\in (H^2\cap\mathcal K^\perp)\oplus\mathcal K.$$
L'application $$h\mapsto P(u^\perp)+cu\in K^\perp\oplus\mathcal K$$ est clairement
un isomorphisme, or c'est $P+c\,\Pi$. La régularité elliptique permet de conclure à l'isomorphisme entre les $H^k$ (voir par exemple \cite{Eichhorn2007} théorème 3.31 page 36).
\end{proof}

\section{Le théorème principal}
Il est maintenant bien connu que l'équation que nous voulons résoudre (\ref{mainequation}) n'est pas elliptique dû à l'invariance
de la courbure par difféomorphisme. Nous allons modifier cette  équation via un terme jauge en s'inspirant
de la m\'ethode de DeTurck. 

Tout d'abord
l'équation (\ref{mainequation}) est équivalente à 
$$
\Ric(g)=E-\frac{\kappa\Tr_g  E+\Lambda}{1+n\kappa}g.
$$ 
Pour toute métrique $g$, $B_{g}(\Ric(g))=0$
par l'identité de Bianchi. Nous définissons donc 
$$
\mathcal B_g(E)=\div_gE+\frac{2\kappa+1}{2(1+\kappa n)}d\Tr_gE=B_g(E)-\frac{(n-2)\kappa}{2(1+\kappa n)}d\Tr_gE,
$$
de sorte que l'identité de Bianchi se traduise ici par
$$
\mathcal B_g(Ein(g))=0.
$$
On définit \cite{Delay:ricciproduit}:
$$
\mathcal F(h,e):=\Ric(g+h)-E+\frac{\kappa\Tr_{g+h}E+\Lambda}{1+\kappa n}{(g+h)}-\mathcal L_{g}\Ein_g^{-1}\mathcal B_{g+h}(E),
$$
où $\Ein_g$ est l'endomorphisme de $T^*M$ associé à $\Ein(g)$, $$E=\Ein(g)+e-\frac12\Pi(h),$$ et $\Pi$ une projection $L^2$ sur un espace de dimension fini
à préciser ultérieurement.

\begin{proposition}\label{Flisse}
Pour $\kappa\neq -1/n$,  $s>\frac n2$ l'application $$
\mathcal F: H^{s+2}(M,\mathcal S_2)\times H^{s+2}(M,\mathcal S_2)\longrightarrow H^{s}(M,\mathcal S_2),
$$
est bien définie et lisse au voisinage de zéro.
\end{proposition}
\begin{proof}
La preuve de cette proposition est renvoyée en appendice,
elle utilise essentiellement le fait  que sous ces hypothèses,
l'espace $H^{s}$ est "uniformément" stable par produit tensoriel
(voir lemme \ref{lemAlgebre}).
\end{proof}

Comme dans \cite{Delay:ricciproduit}, définissons l'opérateur
$$
\begin{array}{lll}
\mathcal P h&:=&\Delta_Lh+\frac{2(n\kappa\tau+\Lambda)}{1+kn}h
+\frac{\kappa}{n(1+\kappa n)}\Big({(n-2)}\Delta \Tr_gh-2 n\tau\Tr_gh\Big)\;g\\
&=&(\Delta_L+{2\kappa R(g)+2\Lambda})h
+\frac{\kappa}{n(1+\kappa n)}\Big({(n-2)}\Delta \Tr_gh-2 n\tau\Tr_gh\Big)\;g.\\
\end{array}
$$
Il sera lié à la différentielle de $\mathcal F$ comme nous allons voir ci-après.
Notons qu'il respecte le scindage $\mathcal S_2=\mathcal G\oplus\mathring {\mathcal S}_2$.
 En particulier si
$u$ est une fonction sur $M$ et $\mathring h$ un champ de 2-tenseurs symétrique sans trace, on a
$$
\mathcal P(ug+\mathring h)=\frac{1}{1+\kappa n}p(u)g+\mathring P(\mathring h),
$$
o\`u
$$
p(u)=(1+2(n-1)\kappa)\Delta u+2\Lambda u,
$$
et 
$$
\mathring P(\mathring h)=\left[\Delta_L+{2\kappa R(g)+2\Lambda}\right]\mathring h.
$$

Pour $u$ une fonction sur $M$ et $\mathring h$ un champ de 2-tenseurs symétrique sans trace, on définit
$$\Pi(ug+\mathring h):=\pi(u)g+\mathring\Pi(\mathring h),$$
o\`u $\pi$ est la projection $L^2$ sur noyau de $p$, et 
$\mathring\Pi$  la projection $L^2$ sur noyau de $\mathring P$.
Ainsi si  $h=ug$ on trouve \cite{Delay:ricciproduit}:
 $$D_h\mathcal F(0,0)(ug)=
\frac1{2}[p(u) +\pi(u)]g-\frac{(n-2)n\kappa}{2(1+\kappa n)}\mathring\Hess \;u,
$$
o\`u $\mathring\Hess \;u$ est la partie sans trace de la hessienne de $u$.
Si $h=\mathring h$ est sans trace, on obtient  \cite{Delay:ricciproduit} :
 $$
D_h\mathcal F(0,0)(\mathring h)
=\frac12\left(\mathring P+ \mathring\Pi\right)\mathring h.
$$
Définissons enfin  l'opérateur agissant sur les 1-formes :
$$
P_H:=\Delta_H+2\kappa R(g)+2\Lambda.
$$
\begin{theorem}\label{theoinvEin}
Soient  $s>n/2$, $\kappa\neq -\frac1n,-\frac1{2(n-1)}$ et  $\Lambda\in\R$. Soit $g$ une métrique Ricci parallèle si $\kappa=0$ et d'Einstein sinon, telle que $\Ein(g)$
est non dégénéré. On suppose que  $p$, $\mathring P$ et $P_H$ sont  Fredholm de $H^2$ dans $L^2$, que le noyau $L^2$ de $p$ est  trivial ou réduit aux constantes, 
et que le noyau de $P_H$ est trivial.  Alors pour tout $ e\in H^{s+2}(M,\mathcal S_2)$ petit,  il existe un unique $h$ proche de zéro dans 
$H^{s+2}(M,\mathcal S_2)$ telle que
$$
 \Ein(g+h)=\Ein(g)+e-\frac12\Pi(h),
$$
De plus l'application $ e\mapsto h$ est lisse au voisinage de zéro entre les espaces de  Hilbert correspondants. 
\end{theorem}

\begin{proof}

Idem à \cite{Delay:ricciproduit}

\end{proof}

\section{Opérateurs  de type Riemann-Christoffel}\label{sec:ssvar}
Nous allons rappeler comment  montrer que l'image de certain opérateurs de 
type Riemann-Christoffel, sont des sous variétés dans $C^\infty$, au voisinage de la  métrique $g$. Définissons un tenseur  $\mathcal Ein$ qui soit  4 fois covariant, ayant les m\^emes propriétés 
algébriques que le tenseur de Riemann, affine en la courbure et dont la trace soit proportionnelle à $\Ein$ \cite{Delay:ricciproduit}:
$$
\mathcal Ein(g)=\Riem(g)+g {~\wedge \!\!\!\!\!\bigcirc ~} (a\Ric(g)+bR(g)g+cg), 
$$
o\`u ${~\wedge \!\!\!\!\!\bigcirc ~}$ est le produit de Kulkarni-Nomizu (\cite{Besse} p. 47),
 $$a\in \R\backslash\left\{\frac{-1}{n-2}\right\},\;\;c=\frac{1+(n-2)a}{2(n-1)}\Lambda,\;\;b=\frac{\kappa[1+a(n-2)]-a}{2(n-1)}.$$
On a alors
$$
\Tr_g\mathcal Ein(g)=[a(n-2)+1]\Ein(g).
$$
La version de type Riemann-Christoffel de $\mathcal Ein(g)$ est définie par
$$
[g^{-1}\mathcal Ein(g)]^i_{klm}:=g^{ij}\mathcal Ein(g)_{jklm}.
$$
Consid\'erons ${\mathcal R}^1_3$, le sous-espace de ${\mathcal T}^1_3$ des
tenseurs v\'erifiants
$$
\tau^i_{ilm}=0,\;\tau^i_{klm}=-\tau^i_{kml},\;
\tau^i_{klm}+\tau^i_{mkl}+\tau^i_{lmk}=0.
$$
On définit l'espace de Fréchet 
$$H^{\infty}=\cap_{k\in\N}H^{k},$$
munit de la famille de semi-normes  $\{\|.\|_{k}\}_{k\in\N}$.
On procède alors de façons similaire à \cite{Delay:etude} pour prouver que 
\begin{theorem}
Sous les conditions du théorème \ref{theoinvEin}, on suppose de plus que le noyau de $\mathcal P$ est trivial,
autrement dit  $ \Pi=0$. Alors l'image de l'application
$$
\begin{array}{lll}
H^{\infty}(M,\mathcal S_2)&\longrightarrow&H^{\infty}(M,\mathcal R_3^1)\\
h&\mapsto &(g+h)^{-1}\mathcal Ein(g+h)-(g)^{-1}\mathcal Ein(g)\\
\end{array}
$$
est une sous-variété lisse au voisinage de zéro.
\end{theorem}


\section{Appendice}
Nous justifions ici la proposition \ref{Flisse} par une preuve  formelle
 (voir  \cite{Delay:etude} pour une preuve similaire 
 particulièrement détaillée). Nous pourrions  omettre cet appendice,
très similaire à celui de \cite{Delay:ricciAE}, qui utilise essentiellement le lemme \ref{lemAlgebre}. Nous  avons choisi de l'adapter afin d'avoir une trame complète de la résolution du problème dans d'autres contextes.\\

Rappelons que la  différence des courbures de Ricci s'exprime en coordonnées locales par
$${
\Ric(g+h)_{jk}-\Ric(g)_{jk}=\nabla_lT^l_{jk}-\nabla_kT^l_{jl}}
+T^p_{jk}T^l_{pl}-T^p_{jl}T^l_{pk},
$$
o\`u
$$
T^k_{ij}=\frac{1}{2}[(g+h)^{-1}]^{ks}(\nabla_ih_{sj}+\nabla_jh_{is}
-\nabla_sh_{ij}).
$$
Nous écrirons donc abusivement
$$
\Ric(g)=\nabla T+TT\;,\;\;\; T=(g+h)^{-1}\nabla h.
$$
Ici nous avons  $h$ petit dans $ H^{s+2}$, $s>\frac n2$. On a alors
$$(g+h)^{-1}=g^{-1}+\widetilde h\;,\;\;\widetilde h\in H^{s+2},$$
avec par inégalité triangulaire et par le lemme \ref{lemAlgebre}
$$
\|\widetilde h\|_{s+2}\leq \sum_{k\in\N}C^{k}\|h\|_{s+2}^{k+1}=\frac{\|h\|_{s+2}}{1-C\|h\|_{s+2}}.
$$
Pour $\|h\|_{s+2}\leq \frac1{2C}$, ce qu'on suppose désormais, on a 
$$
\|\widetilde h\|_{s+2}\leq 2\|h\|_{s+2}.
$$
On obtient alors, en utilisant encore le lemme \ref{lemAlgebre}, et en omettant dorénavant les constantes 
$$
T=(g^{-1}+\widetilde h)\nabla h\in H^{s+1}\;,\;\;\|T\|_{s+1}\leq \|h\|_{s+2},
$$
et 
$$
\nabla T\in H^{s}\;,\;\;\|\nabla T\|_{s}\leq \|T\|_{s+1}\leq \|h\|_{s+2},
$$
d'o\`u, toujours en utilisant  le lemme \ref{lemAlgebre},
$$
\Ric(g+h)-\Ric(g) \in H^{s}\;,\;\;\|\Ric(g+h)-\Ric(g)\|_{s}\leq \|h\|_{s+2}.
$$
Étudions maintenant l'opérateur de Bianchi 
$$
\mathcal B_{(g+h)}(E)=\div_{(g+h)}E+\frac{2\kappa+1}{2(1+\kappa n)}d\Tr_{(g+h)}E,
$$
que nous écrirons encore abusivement 
$$
\mathcal B_{(g+h)}(E)=(g+h)^{-1}(\nabla  E+T E)+\nabla [(g+h)^{-1}E].
$$
Compte tenu des calculs précédent et du fait que $E=\Ein(g)+e$ (avec $\nabla \Ein(g)=0$ par hypothèse),
on a 
$$
\mathcal B_{g+h}(E)=(g^{-1}+\widetilde h)[\nabla  e+T (\Ein(g)+e)]+\nabla [g^{-1}e+\widetilde h (\Ein(g)+e)].
$$
On estime alors comme précédemment  
$$
\mathcal B_{g+h}(E)\in H^{s+1}\;,\;\;\|\mathcal B_{g+h}(E)\|_{s+1}\leq (\|h\|_{s+2}+\|e\|_{s+2}),
$$
et
$
\mathcal L_g\Ein_g^{-1}\mathcal B_{g+h}(E)\in H^{s},$ $$\|\mathcal L_g\Ein_g^{-1}\mathcal B_{g+h}(E)\|_{s}\leq
\|\mathcal B_{g+h}(E)\|_{s+1}\leq (\|h\|_{s+2}+\|e\|_{s+2}).
$$
Il reste a estimer un terme d'ordre zéro :
$$
Z:=\frac{\kappa\Tr_{g+h}E+\Lambda}{1+\kappa n}{(g+h)}-E+\Ric(g)
$$
On écrit encore formellement, en se souvenant ici que le premier "produit" est une trace,
$$
\begin{array}{lll}
Z&=&\displaystyle{\frac{\kappa(g^{-1}+\widetilde h)(\Ein(g)+e)+\Lambda}{1+n\kappa}(g+h)-(\Ein(g)-\Ric(g)+e)}\\
&=&\displaystyle{\frac{\kappa g^{-1}e+\kappa\widetilde h(\Ein(g)+e)+\Lambda+\kappa\Tr_g\Ein(g))}{1+n\kappa}(g+h)}\\
&&\hspace{7cm}-(\Ein(g)-\Ric_g+e).\\
\end{array}
$$
En développant, on remarque que le terme "constant":
$$
\frac{\Lambda+\kappa\Tr_g\Ein(g)}{1+n\kappa}g-(\Ein(g)-\Ric(g))$$
est nul et que l'on peut estimer
comme auparavant, pour $k\neq -1/n$,
$$
Z\in H^{s}\;,\;\;\|Z\|_{s}\leq \|Z\|_{s+2}\leq(\|h\|_{s+2}+\|e\|_{s+2}).
$$
%

\providecommand{\bysame}{\leavevmode\hbox to3em{\hrulefill}\thinspace}
\providecommand{\MR}{\relax\ifhmode\unskip\space\fi MR }
\providecommand{\MRhref}[2]{%
  \href{http://www.ams.org/mathscinet-getitem?mr=#1}{#2}
}
\providecommand{\href}[2]{#2}

\end{document}